\documentclass[12pt]{article}

\usepackage{amssymb,amsmath,amsthm, graphics, color,epsfig}
\usepackage{verbatim}
\usepackage{pstricks,pstricks-add,pst-math,pst-xkey,subfigure}

\makeatletter \@addtoreset{equation}{section} \makeatother

\newcommand{\eff}{{\rm \scriptscriptstyle eff}}

\newcommand{\be}{\begin{equation}}
\newcommand{\ee}{\end{equation}}
\newcommand{\ba}{\begin{array}}
\newcommand{\ea}{\end{array}}
\newcommand{\disp}{\displaystyle}

\begin{document}



\newtheorem{theorem}{Theorem}[section]
\newtheorem{lemma}{Lemma}[section]
\newtheorem{definition}{Definition}[section]
\newtheorem{remark}{Remark}[section]
\newtheorem{corollary}{Corollary}[section]
\newtheorem{proposition}{Proposition}[section]

\renewcommand{\abstractname}{Abstract}
\renewcommand{\refname}{References}

\newcommand{\ve}{\varepsilon}
\newcommand{\he}{{\mathcal{H}^\varepsilon}}
\newcommand{\dx}[1]{\frac{d #1}{d x_1}}
\newcommand{\ddx}[1]{\frac{d^2 #1}{d x_1^2}}


\title{Localization effect for a spectral problem in a perforated domain with Fourier boundary conditions}

\author{
V.~Chiad\`{o} Piat \thanks{Politecnico di Torino, Corso Duca degli Abruzzi 24, 10129 Torino, Italy (vchiado@polito.it)} \and
I. Pankratova \thanks{Narvik University College, Postbox 385, 8505 Narvik, Norway (iripan@hin.no)} \and
A. Piatnitski \thanks{Narvik University College, Postbox 385, 8505
Narvik, Norway; Lebedev Physical Institute RAS, Leninski ave., 53, 119991 Moscow, Russia (andrey@sci.lebedev.ru)}}

\date{}
\maketitle

\abstract{We consider a homogenization of elliptic spectral problem stated in a perforated domain, Fourier boundary conditions being imposed on the boundary of perforation. The presence of a locally periodic coefficient in the boundary operator gives rise to the effect of a localization of the eigenfunctions. Moreover, the limit behaviour of the lower part of the spectrum can be described in terms of an auxiliary harmonic oscillator operator.
We describe the asymptotics of the eigenpairs and derive the estimates for the rate of convergence.}

\bigskip

\section{Introduction}

The paper deals with a spectral problem for a second order
divergence form elliptic operator in a periodically perforated bounded domain in $\mathbb R^d$. Assuming that on the perforation border a homogeneous Fourier boundary condition is stated, and that the coefficient of the boundary operator is a function of "slow" argument, we arrive at the following eigenvalue problem
\begin{equation}
\label{intr_or-prob}
\left\{
\begin{array}{lcr}
\displaystyle
- {\rm div}(a(x/\ve) \nabla u^\ve(x)) = \lambda^\ve u^\ve(x), \quad \hfill x \in \Omega_\ve,
\\[3mm]
a(x/\ve) \nabla u^\ve(x)\cdot n = - q(x)u^\ve(x), \quad \hfill x \in \Sigma_\ve,
\\[3mm]
u^\ve(x)=0, \quad \hfill x \in \partial \Omega;
\end{array}
\right.
\end{equation}
here $\ve$ is a small positive parameter defined as a microstructure period.

We impose some natural regularity and connectedness conditions on the perforated domain $\Omega_\ve$, as well as usual periodicity and uniform ellipticity conditions on
the matrix $a(y)$. These conditions are specified in detail in the next section.

Our crucial assumptions are
\begin{itemize}
\item $q\in C^2(\overline\Omega)$, and
$q(x)\geq q_0>0$ in $\overline\Omega$.
\item The function $q$ has only one global minimum point
in $\overline\Omega$. The global minimum is attained at an interior point of $\Omega$.
\item The Hessian matrix $\partial^2 q/\partial x^2$ evaluated at
the minimum point is positive definite.
\end{itemize}

Under the first two assumptions the localization phenomenon holds.
Namely, for any $k\in\mathbb N$ the $k$-th eigenfunction of problem (\ref{intr_or-prob}) is asymptotically localized, as $\ve\to0$, in a small neighbourhood of the minimum point. In particular, the properly normalized principal eigenfunction converges to a $\delta$-function supported at the minimum point.

In the paper, assuming that all the above conditions are fulfilled, we construct the first two leading terms of the asymptotic expansions for the $k$-th eigenpair, $k=1,2,\dots$.

These asymptotic expansions have a number of interesting features. First of all, the mentioned expansions are in integer powers
of $\ve^{1/4}$.
Then, the localization takes place in the scale $\ve^{1/4}$.
In this scale the leading term of the asymptotic expansion
for the $k$-th eigenfunction proved to be the $k$-th eigenfunction of an auxiliary harmonic oscillator operator.

If $q\in C^3(\overline\Omega)$, then we also obtain the estimates for the rate of convergence.

We suppose that $q$ does not oscillate just for presentation simplicity. The techniques developed in the paper also apply
to the case of locally periodic coefficients $q=q(x,x/\ve)$,
 \
$a=a(x,x/\ve)$ with $q(x,y)$ and $a(x,y)$ being periodic in $y$, see Remark \ref{remark-loc-per} and Theorem \ref{Th-short-loc-per} below.

Previously, the localization phenomenon in spectral problems
has been observed in several mathematical works.
In \cite{AlPi-2002} the operator with a large locally periodic potential has been considered. The localization appeared due
to the presence of a large factor in the potential and the fact that the operator coefficients depend on slow variable.

In \cite{FrSo-2009} the Dirichlet spectral problem for the Laplacian in a thin 2D strip of slowly varying thickness
has been studied. Here the localization has been observed
in the vicinity of the point of maximum thickness. The large
parameter is the first eigenvalue of 1D Laplacian in the cross-section. This eigenvalue grows to infinity because the
thickness of the strip asymptotically vanishes.

In the mentioned works, under natural non-degeneracy conditions, the asymptotics of the eigenpairs was
described in terms of the spectrum of an appropriate harmonic
oscillator operator.
However, the localization scale was of order
$\sqrt{\ve}$ with $\ve$ being the microscopic length scale.

The localization in the scale $\ve^{1/4}$ that is observed
in the present paper, is not standard. It should also be noted
that although the operators in (\ref{intr_or-prob}) do not contain a large parameter, such a parameter is presented implicitly
because $(d-1)$-dimensional volume of the perforation surface
tends to infinity.

The homogenization of spectral problem (\ref{intr_or-prob}) with a constant or periodic functions $a$ and $q$ has been addressed in \cite{Pastukh-01}.

Spectral problems in perforated domains with Dirichlet and Neumann boundary condition at the perforation border are now well studied. There is a vast literature on the topic, see, for instance, \cite{Van-81}, \cite{OlYoSh}.

In the paper we combine asymptotic expansion techniques with
various variational and compactness arguments and scaled
trace and Poincar\'e type inequalities.

\section{Problem statement}
We start by describing the geometry of the domain. Let $K= [0,1)^d$ and $E \subset \mathbb{R}^d$  be a $K$-periodic, open, connected set with a Lipschitz boundary $\Sigma$; the complement $\mathbb{R}^d \setminus E$ is denoted by $B$. We also assume that $K \cap E$ is a connected set, and $K \cap B \Subset K$, so that $B = \mathbb{R}^d \setminus E$ consists of disjoint components. In what follows, $Y = K \cap E$ denotes the periodicity cell, and $\Sigma^0=K\cap \partial B=K \cap \Sigma$ the boundary of the inclusion.
The symbols $|Y|_d$ and $|\Sigma^0|_{d-1}$ stand for the measures of $Y$ and the $(d-1)$-dimensional surface measure of $\Sigma^0$, respectively.

For every $i \in \mathbb{Z}^d$ we denote $Y_\ve^i = \ve(i + Y)$, $\Sigma_\ve^i = \ve \Sigma \bigcap Y_\ve^i$, and $B_\ve^i = \ve B \bigcap Y_\ve^i$. Given $\Omega$, a bounded domain in $\mathbb{R}^d$ with a Lipschitz boundary $\partial \Omega$, we introduce the perforated domain
$$
\Omega_\ve = \Omega \setminus \bigcup \limits_{i \in I_\ve} B_\ve^i, \quad I_\ve = \{i \in \mathbb{Z}^d: \,\, Y_\ve^i \subset \Omega\}.
$$
Notice that $\Omega_\ve$ remains connected, the perforation does not intersect the boundary $\partial \Omega$, and
$$
\partial \Omega_\ve = \partial \Omega \bigcup \Sigma_\ve, \quad \Sigma_\ve = \bigcup \limits_{i \in I_\ve} \Sigma_\ve^i.
$$
In the perforated domain $\Omega_\ve$ we consider the following spectral problem:
\begin{equation}
\label{or-prob}
\left\{
\begin{array}{lcr}
\displaystyle
- {\rm div}(a^\ve(x) \nabla u^\ve(x)) = \lambda^\ve u^\ve(x), \quad \hfill x \in \Omega_\ve,
\\[3mm]
a^\ve(x) \nabla u^\ve(x)\cdot n = - q(x)u^\ve(x), \quad \hfill x \in \Sigma_\ve,
\\[3mm]
u^\ve(x)=0, \quad \hfill x \in \partial \Omega.
\end{array}
\right.
\end{equation}
Here $\ve$ is a small positive parameter, $a^\ve(x)= a(x/\ve)$ with $a(y)$ being a $d \times d$ matrix, $n$ is an outward unit normal; the usual scalar product in $\mathbb{R}^d$ is denoted by $"\cdot"$.

\begin{figure}
\centering
\psset{xunit=0.9cm,yunit=0.9cm,algebraic=true,dotstyle=*,dotsize=3pt 0,linewidth=0.8pt,arrowsize=3pt 2,arrowinset=0.25}
\begin{pspicture*}(-3.07,-0.96)(3.85,4.47)
\rput{169.41}(0.45,1.72){\psellipse(0,0)(3.24,2.48)}
\pscircle(-0.75,3.75){0.17}
\pscircle(-0.25,3.75){0.17}
\pscircle(0.25,3.75){0.17}
\pscircle(0.75,3.75){0.17}
\pscircle(1.25,3.75){0.17}
\pscircle(-0.75,3.25){0.17}
\pscircle(-0.25,3.25){0.17}
\pscircle(0.25,3.25){0.17}
\pscircle(0.75,3.25){0.17}
\pscircle(1.25,3.25){0.17}
\pscircle(1.75,3.25){0.17}
\pscircle(2.25,3.25){0.17}
\pscircle(-1.25,3.25){0.17}
\pscircle(-1.75,3.25){0.17}
\pscircle(-1.75,2.75){0.17}
\pscircle(-1.25,2.75){0.17}
\pscircle(-0.75,2.75){0.17}
\pscircle(-0.25,2.75){0.17}
\pscircle(0.25,2.75){0.17}
\pscircle(0.75,2.75){0.17}
\pscircle(1.25,2.75){0.17}
\pscircle(1.75,2.75){0.17}
\pscircle(2.25,2.75){0.17}
\pscircle(2.75,2.75){0.17}
\pscircle(2.75,2.25){0.17}
\pscircle(-1.75,2.25){0.17}
\pscircle(-1.25,2.25){0.17}
\pscircle(-0.75,2.25){0.17}
\pscircle(-0.25,2.25){0.17}
\pscircle(0.25,2.25){0.17}
\pscircle(0.75,2.25){0.17}
\pscircle(1.25,2.25){0.17}
\pscircle(1.75,2.25){0.17}
\pscircle(2.25,2.25){0.17}
\pscircle(2.75,2.25){0.17}
\pscircle(-1.75,1.75){0.17}
\pscircle(-1.25,1.75){0.17}
\pscircle(-0.75,1.75){0.17}
\pscircle(-0.25,1.75){0.17}
\pscircle(0.25,1.75){0.17}
\pscircle(0.75,1.75){0.17}
\pscircle(1.25,1.75){0.17}
\pscircle(1.75,1.75){0.17}
\pscircle(2.25,1.75){0.17}
\pscircle(2.75,1.75){0.17}
\pscircle(3.25,1.75){0.17}
\pscircle(-0.25,1.25){0.17}
\pscircle(-0.75,1.25){0.17}
\pscircle(-1.25,1.25){0.17}
\pscircle(-1.75,1.25){0.17}
\pscircle(0.25,1.25){0.17}
\pscircle(0.75,1.25){0.17}
\pscircle(1.25,1.25){0.17}
\pscircle(1.75,1.25){0.17}
\pscircle(2.25,1.25){0.17}
\pscircle(2.75,1.25){0.17}
\pscircle(3.25,1.25){0.17}
\pscircle(-1.75,0.75){0.17}
\pscircle(-1.25,0.75){0.17}
\pscircle(-0.75,0.75){0.17}
\pscircle(-0.25,0.75){0.17}
\pscircle(0.25,0.75){0.17}
\pscircle(0.75,0.75){0.17}
\pscircle(1.25,0.75){0.17}
\pscircle(1.75,0.75){0.17}
\pscircle(2.25,0.75){0.17}
\pscircle(2.75,0.75){0.17}
\pscircle(2.75,0.25){0.17}
\pscircle(2.25,0.25){0.17}
\pscircle(-0.75,0.25){0.17}
\pscircle(-1.25,0.25){0.17}
\pscircle(-0.25,0.25){0.17}
\pscircle(0.25,0.25){0.17}
\pscircle(0.75,0.25){0.17}
\pscircle(1.25,0.25){0.17}
\pscircle(1.75,0.25){0.17}
\pscircle(-0.75,-0.25){0.17}
\pscircle(-0.25,-0.25){0.17}
\pscircle(0.25,-0.25){0.17}
\pscircle(0.75,-0.25){0.17}
\pscircle(1.25,-0.25){0.17}
\pscircle(1.75,-0.25){0.17}
\psline[linewidth=0.4pt](-1.87,2.9)(-2.45,3.74)
\psline[linewidth=0.4pt](-1.4,2.9)(-2.45,3.74)
\psline[linewidth=0.4pt](-1.92,2.19)(-2.45,3.74)
\rput[tl](-2.5,4.11){$\Sigma_\ve$}
\rput[tl](3.16,3.43){$\Omega_\ve$}
\pscircle(-2.25,2.75){0.17}
\pscircle(-2.25,2.25){0.17}
\pscircle(-2.25,1.75){0.17}
\pscircle(-2.25,1.25){0.17}
\end{pspicture*}
\caption{Domain $\Omega_\ve$}
\end{figure}


In the sequel we assume that the following conditions hold true:
\begin{itemize}
\item[\textbf{(H1)}]
$a(y)$ is a real symmetric $d \times d$ matrix satisfying the uniform ellipticity condition
$$
\sum\limits_{i,j =1}^d a_{ij}(y) \xi_i \xi_j \ge \Lambda\, |\xi|^2, \quad \xi \in \mathbb{R}^d,
$$
for some $\Lambda>0$.
\item[\textbf{(H2)}]
The coefficients  $a_{ij}(y)$ are  $Y$-periodic and, moreover,  $a_{ij}(y) \in L^\infty(\mathbb{R}^d)$.
\item[\textbf{(H3)}]
The function $q(x) \in C^3(\mathbb{R}^d)$ is positive.
\item[\textbf{(H4)}]
The function $q(x)$ has a unique global minimum attained at $x=0 \in \Omega$. Moreover, in the vicinity of $x=0$
$$
q(x)= q(0) + \frac{1}{2}\, x^T\, H(q)\, x + o(|x|^2),
$$
with the positive definite Hessian matrix $H(q)$.
\end{itemize}

It is convenient to introduce the notation
$$
H_0^1(\Omega_\ve, \partial \Omega) = \{ u \in H^1(\Omega_\ve): \,\, u =0 \,\, \mbox{\rm on} \,\, \partial \Omega\}.
$$
The weak formulation of spectral problem \eqref{or-prob} reads: find $\lambda^\ve \in \mathbb{C}$ (eigenvalues) and $u^\ve \in H_0^1(\Omega_\ve, \partial \Omega)$, $u^\ve \neq 0$, such that
\begin{equation}
\label{weak-or-prob}
\int \limits_{\Omega_\ve} a^\ve \nabla u^\ve \cdot \nabla v\, dx +
\int \limits_{\Sigma_\ve} q\, u^\ve\, v\, d\sigma =
\lambda^\ve \int \limits_{\Omega_\ve} u^\ve\, v\, dx, \quad v \in H_0^1(\Omega).
\end{equation}
\begin{lemma}
\label{lemma-struc-spectrum}
For any $\ve >0$, the spectrum of problem \eqref{weak-or-prob} is real and consists of a countable set of points
$$
0 < \lambda_1^\ve < \lambda_2^\ve \le \cdots \le \lambda_j^\ve \le \cdots \to +\infty.
$$
Every eigenvalue has a finite multiplicity. The corresponding eigenfunctions normalized by
$$
\int \limits_{\Omega_\ve} u_i^\ve\, u_j^\ve\, dx  = \delta_{ij},
$$
form an orthonormal basis in $L^2(\Omega_\ve)$.
\end{lemma}
We omit the proof of Lemma~\ref{lemma-struc-spectrum} which is classical.

Under the assumptions \textbf{(H1)-(H4)} we study the asymptotic behaviour of eigenpairs $(\lambda^\ve, u^\ve)$, as $\ve \to 0$.

To avoid excessive technicalities for the moment, we state our main result in a slightly reduced form, without specifying the rate of convergence. For the detailed formulation of the main result see Theorem~\ref{Th-main-full}.
\begin{theorem}
\label{Th-short}
Let conditions \textbf{(H1)-(H4)} be fulfilled. If $(\lambda_j^\ve, u_j^\ve)$ stands for the $j$th eigenpair of problem \eqref{or-prob}, then
for any $j$, the following representation takes place:
$$
\lambda_j^\ve = \frac{1}{\ve}\, \frac{|\Sigma^0 |_{d-1}}{|Y|_d}\, q(0) + \frac{\mu_j^\ve}{\sqrt{\ve}},
\quad u_j^\ve(x) = v_j^\ve \big(\frac{x}{\ve^{1/4}}\big),
$$
where $(\mu_j^\ve, v_j^\ve(z))$ are such that
\begin{itemize}
\item
$\mu_j^\ve$ converges, as $\ve \to 0$, to the $j$th eiegnvalue $\mu_j$ of the effective spectral problem
\be
\label{eff-prob-intro}
-{\rm div}(a^\eff \nabla v) + (z^T Q z)\, v = \mu\, v, \quad
v \in L^2(\mathbb{R}^d),
\ee
where $a^\eff$ is a positive definite matrix (see \eqref{a^eff}); $Q$ is defined by
$$
Q = \frac{1}{2}\, \frac{|\Sigma^0|_{d-1}}{|Y|_d}\,H(q),
$$
with $H(q)$ being the Hessian matrix of $q$ at $x=0$.
.
\item
If $\mu_j$ is a simple eigenvalue, then, for small $\ve$, $\mu_j^\ve$ is also simple, and the convergence of the corresponding eigenfunctions (extended to the whole $\mathbb{R}^d$) holds
$$
\|v_j^\ve - v_j\|_{L^2(\mathbb{R}^d)} \to 0, \quad \ve \to 0.
$$
\end{itemize}
\end{theorem}

\begin{remark}
\label{remark-loc-per}
Theorem~\ref{Th-short} can be generalized to the case of locally periodic coefficients in (\ref{or-prob}).

Namely, let us consider the following problem:
\begin{equation}
\label{loc-per}
\left\{
\begin{array}{lcr}
\displaystyle
- {\rm div}(a^\ve(x) \nabla u^\ve(x)) = \lambda^\ve u^\ve(x), \quad \hfill x \in \Omega_\ve,
\\[3mm]
a^\ve(x) \nabla u^\ve(x)\cdot n = - q^\ve(x)\,u^\ve(x), \quad \hfill x \in \Sigma_\ve,
\\[3mm]
u^\ve(x)=0, \quad \hfill x \in \partial \Omega
\end{array}
\right.
\end{equation}
with
$$
a^\ve(x) = a(x,x/\ve), \qquad q^\ve(x) = q(x,x/\ve).
$$
Assume that
\begin{itemize}
\item
$a_{ij}(x,y)$ and $q(x,y)$ are $Y$-periodic in $y$
functions such that
$a_{ij}(x,y)$, $q(x,y) \in C^{2,\alpha}(\mathbb{R}^d; C^\alpha({\overline{Y}}))$ with some $\alpha>0$.
\item
The matrix $a(x,y)$ satisfies the uniform elipticity condition.
\item
The local average of $q$ defined by
$$
\bar{q}(x) = \frac{1}{|\Sigma^0|_{d-1}} \, \int \limits_{\Sigma^0} q(x,y)\, d\sigma_y,
$$
admits its global minimum at $x=0$.
\item In the vicinity of $x=0$
$$
\bar{q}(x) = \bar{q}(0) + \frac{1}{2} \, x^T \,H(\bar{q})\,x + o(|x|^2)
$$
with the positive definite Hessian matrix $H(\bar{q})$.
\item $x=0$ is the only global minimum point of $\bar{q}$
in $\overline\Omega$.
\end{itemize}
Then the following convergence result holds.
\begin{theorem}
\label{Th-short-loc-per}
If $(\lambda_j^\ve, u_j^\ve)$ stands for the $j$th eigenpair of problem \eqref{loc-per}, then
for any $j$, the following representation takes place:
$$
\lambda_j^\ve = \frac{1}{\ve}\, \frac{|\Sigma^0|_{d-1}}{|Y|_d}\, \bar{q}(0) + \frac{\mu_j^\ve}{\sqrt{\ve}},
\quad u_j^\ve(x) = v_j^\ve \big(\frac{x}{\ve^{1/4}}\big),
$$
where $(\mu_j^\ve, v_j^\ve(z))$ are such that
\begin{itemize}
\item
$\mu_j^\ve$ converges, as $\ve \to 0$, to the $j$th eiegnvalue $\mu_j$ of the effective spectral problem
\be
\label{eff-prob-loc-per}
-{\rm div}(a^\eff \nabla v) + (z^T P z)\, v = \mu\, v, \quad
v \in L^2(\mathbb{R}^d),
\ee
where
$$
P= \frac{1}{2} \frac{|\Sigma^0|_{d-1}}{|Y|_d}\, H(\bar{q})
$$
and $a^\eff$ is a positive definite matrix defined by
$$
a_{ij}^\eff = \frac{1}{|Y|_d}\,\int \limits_Y a_{ik}(0,\zeta) (\delta_{kj} + \partial_k N_j(\zeta))\, d\zeta,
$$
with the functions $N_j$ solving auxiliary cell problems
$$
\left\{
\begin{array}{lcr}
- {\rm div}_{\zeta} (a(0,\zeta) \nabla_{\zeta} N_k(\zeta)) = {\rm div}_{\zeta} a_{\cdot k}(0, \zeta), \quad k = 1,...,d, \quad \zeta \in Y,
\\[2mm]
a(0,\zeta)\nabla_\zeta N_k \cdot n= - a_{i k}(0,\zeta) n_i, \quad \zeta \in \Sigma^0,
\\[2mm]
N_k(\zeta)\in H_{\#}^1(Y),
\end{array}
\right.
$$
\item
If $\mu_j$ is a simple eigenvalue, then, for small $\ve$, $\mu_j^\ve$ is also simple, and the convergence of the corresponding eigenfunctions (extended to the whole $\mathbb{R}^d$) holds
$$
\|v_j^\ve - v_j\|_{L^2(\mathbb{R}^d)} \to 0, \quad \ve \to 0.
$$
\end{itemize}
\end{theorem}
\end{remark}
\section{Proof of Theorem~\ref{Th-short}}
\subsection{Preliminaries. Estimates for $\lambda_1^\ve$}
In this section we estimate the first eigenvalue $\lambda_1^\ve$ of problem \eqref{or-prob}.
To this end we use the variational representation for $\lambda_1^\ve$. Let us recall that,
due to the classical min-max principle (see, for example, \cite{Mikh-PDE}),
\begin{equation}
\label{var-lambda_1^eps}
\lambda_1^\ve = \inf \limits_{v\in H_0^1(\Omega_\ve, \partial \Omega)}
\,\, \frac{\int \limits_{\Omega_\ve} a^\ve \nabla v\cdot \nabla v\, dx +
\int \limits_{\Sigma_\ve} q\, (v)^2\, d\sigma}{\|v\|_{L^2(\Omega_\ve)}^2}.
\end{equation}
\begin{lemma}
\label{lemma-est-lambda^eps}
The first eigenvalue of the spectral problem \eqref{or-prob} satisfies the estimate
$$
\frac{1}{\ve}\, \frac{|\Sigma^0|_{d-1}}{|Y|_d}\, q(0) + O(1)\,
\le\, \lambda_1^\ve \,
\le \, \frac{1}{\ve}\, \frac{|\Sigma^0|_{d-1}}{|Y|_d}\, q(0) + O(\ve^{-1/2}), \quad \ve \to 0.
$$
\end{lemma}
\begin{proof}
We start by proving the estimate from below. By \eqref{var-lambda_1^eps},
$$
\lambda_1^\ve
\ge \inf \limits_{{v\in H_0^1(\Omega_\ve, \partial \Omega)}\atop{\|v\|_{L^2(\Omega_\ve)}=1}}
\,\, \Big\{\int \limits_{\Omega_\ve} a^\ve \nabla v \cdot \nabla v\, dx +
q(0)\, \int \limits_{\Sigma_\ve} (v)^2\, d\sigma\Big\}.
$$
The last infimum is attained on the first eigenfunction of the following spectral problem
$$
\left\{
\begin{array}{lcr}
\displaystyle
- {\rm div}(a^\ve(x) \nabla w^\ve(x)) = \nu^\ve w^\ve(x), \quad \hfill x \in \Omega_\ve,
\\[3mm]
a^\ve(x) \nabla w^\ve(x)\cdot n = - q(0)w^\ve(x), \quad \hfill x \in \Sigma_\ve,
\\[3mm]
w^\ve(x)=0, \quad \hfill x \in \partial \Omega.
\end{array}
\right.
$$
It has been proven in \cite{Pastukh-01} that the first eigenvalue of this problem admits the following asymptotics:
$$
\nu_1^\ve = \frac{1}{\ve}\, \frac{|\Sigma^0|_{d-1}}{|Y|_d}\, q(0) + O(1), \qquad \ve \to 0.
$$
Thus,
$$
\lambda_1^\ve \ge \frac{1}{\ve}\, \frac{|\Sigma^0|_{d-1}}{|Y|_d}\, q(0) + O(1), \qquad \ve \to 0.
$$

We proceed to the derivation of the upper bound for $\lambda_1^\ve$. Choosing $v\in C_0^\infty(\Omega)$ as a test function in \eqref{var-lambda_1^eps}, one can obtain a rough estimate
\begin{equation}
\label{rough-est-lambda_1^eps}
\lambda_1^\ve \le \tilde{C}\, \ve^{-1},
\end{equation}
with a constant $\tilde{C}$ independent of $\ve$. To specify $\tilde{C}$ one should choose a "smarter" test function. Let us take $v\in C_0^\infty(\mathbb{R}^d)$, $\|v\|_{L^2(\mathbb{R}^d)}=1$, and choose $v(x/\ve^\alpha)$ as a test function in \eqref{var-lambda_1^eps}, $0<\alpha<1/2$. Note that if ${\rm supp}\, v \subset B_R(0)$, for some $R>0$, then ${\rm supp}\, v(x/\ve^\alpha) \subset B_{\ve^\alpha R}(0)$. Then we obtain
$$
\lambda_1^\ve \le
\frac{\int \limits_{\Sigma_\ve}q(x)\, \big|v\big(\frac{x}{\ve^\alpha}\big)\big|^2\, d\sigma + O(\ve^{-2\alpha}\, \ve^{d \alpha})}{\int \limits_{\Omega_\ve}\big|v\big(\frac{x}{\ve^\alpha}\big)\big|^2\, dx}.
$$
Taking into account assumption $\bf(H4)$ and using Lemma~\ref{lemma-1}, one has
$$
\lambda_1^\ve \le
\frac{\frac{1}{\ve}\frac{|\Sigma^0|_{d-1}}{|Y|_d} \,q(0)\, \int \limits_{\Omega_\ve}\big|v\big(\frac{x}{\ve^\alpha}\big)\big|^2\, dx +O(\ve^{2\alpha-1}\, \ve^{d \alpha})+  O(\ve^{-2\alpha}\, \ve^{d \alpha})}{\int \limits_{\Omega_\ve}\big|v\big(\frac{x}{\ve^\alpha}\big)\big|^2\, dx}.
$$
Notice that the best estimate is obtained for $\alpha = 1/4$. Finally,
\begin{equation}
\label{3}
\lambda_1^\ve \le \frac{1}{\ve}\,\frac{|\Sigma^0|_{d-1}}{|Y|_d} \, q(0) + O(\ve^{-1/2}), \quad \ve \to 0.
\end{equation}
\end{proof}

\begin{remark}
When deriving the upper bound for $\lambda_1^\ve$, we used a test function which is concentrated at $x=0$. Namely, the test function of the form $v(\ve^{-1/4}x)$. This observation turns out  very helpful for the construction of the asymptotics of eigenpairs $(\lambda^\ve, u^\ve)$.
\end{remark}

The next definition explains the notion of concentration.

\begin{definition}\label{def-concentr}
We say that a family $\{w_\ve(x)\}_{\ve>0}$ with $0<c_1\le\|w_\ve\|_{L^2(\Omega_\ve)}\le c_2$ is concentrated at $x_0$, as $\ve\to0$, if for any $\gamma>0$ there is $\ve_0>0$
such that
$$
\int \limits_{\Omega_\ve \setminus B_\gamma(x_0)} |w_\ve|^2\, dx < \gamma, \qquad\hbox{for all }\ve \in (0,\ve_0).
$$
Here $B_\gamma(x_0)$ is a ball of radius $\gamma$ centered at $x_0$.
\end{definition}
\begin{lemma}
\label{lemma-concentr}
The first eigenfunction $u_1^\ve$ of problem \eqref{or-prob} is concentrated in the sense of Definition~\ref{def-concentr} at the minimum point of $q(x)$, that is at $x=0$.
\end{lemma}
\begin{proof}
Assume that $u_1^\ve$, normalized by $\|u_1^\ve\|_{L^2(\Omega_\ve)}=1$, is not concentrated at $x=0$. Then, there exists $\gamma >0$ such that, for any $\ve_0$,  we have
\begin{equation}
\label{1}
\int \limits_{\Omega_\ve \setminus B_\gamma(0)} |u_1^\ve|^2\, dx > \gamma
\end{equation}
for some $\ve < \ve_0$.

Estimate \eqref{rough-est-lambda_1^eps} together with \eqref{var-lambda_1^eps} imply the estimate
$$
\int \limits_{\Omega_\ve} |\nabla u_1^\ve|^2\, dx \le C\, \ve^{-1}.
$$
Then, using Lemma~\ref{lemma-1}, we obtain
$$
\begin{array}{l}
\displaystyle
\lambda_1^\ve =
\int \limits_{\Omega_\ve} a^\ve \nabla u_1^\ve\cdot \nabla u_1^\ve\, dx + \frac{1}{\ve}\,\frac{|\Sigma^0|_{d-1}}{|Y|_d} \int \limits_{\Omega_\ve} q\, |u_1^\ve|^2\, dx + O(\ve^{-1/2})
\cr\cr
\displaystyle
\ge \frac{1}{\ve}\,\frac{|\Sigma^0|_{d-1}}{|Y|_d} \min \limits_{\Omega_\ve \setminus B_\gamma(0)} q \, \int \limits_{\Omega_\ve \setminus B_\gamma(0)} |u_1^\ve|^2\, dx
\cr\cr
\displaystyle
+\frac{1}{\ve}\,\frac{|\Sigma^0|_{d-1}}{|Y|_d} \Big\{ q(0)\int \limits_{\Omega_\ve \cap B_\gamma(0)} |u_1^\ve|^2\, dx
+\int \limits_{\Omega_\ve \cap B_\gamma(0)}(q(x) - q(0)) |u_1^\ve|^2\, dx\Big\} +  O(\ve^{-1/2})
\end{array}
$$
Since $x=0$ is the global minimum point of $q(x)$, then
$$
\begin{array}{l}
\displaystyle
\lambda_1^\ve \ge
\frac{1}{\ve}\,\frac{|\Sigma^0|_{d-1}}{|Y|_d} \,\Big\{ \min \limits_{\Omega_\ve \setminus B_\gamma(0)} q \,
\int \limits_{\Omega_\ve \setminus B_\gamma(0)} |u_1^\ve|^2\, dx
+ q(0)\int \limits_{\Omega_\ve \cap B_\gamma(0)} |u_1^\ve|^2\, dx\Big\}
+  O(\ve^{-1/2}).
\end{array}
$$
By \eqref{1},
\begin{equation}
\label{2}
\lambda_1^\ve \ge \frac{1}{\ve}\,\frac{|\Sigma^0|_{d-1}}{|Y|_d} \, q(0)
+ \frac{1}{\ve}\,\frac{|\Sigma^0|_{d-1}}{|Y|_d} \, \Big(\min \limits_{\Omega_\ve \setminus B_\gamma(0)} q \, - q(0)\Big)\,\gamma
+  O(\ve^{-1/2}),
\end{equation}
that contradicts \eqref{3}. Lemma is proved.
\end{proof}

\begin{remark}
The min-max principle allows us to compare the eigenvalues of Dirichlet, Neumann and Fourier spectral problems. Namely, denote by $\lambda_{D, k}^\ve$ the $k$th eigenvalue of the Dirichlet problem ($u^\ve =0$ on $\Sigma_\ve$), and by $\lambda_{N, k}^\ve$ the $k$th eiegnvalue of the Neumann problem (the case $q=0$ in \eqref{or-prob}). Then, one can see that
\begin{equation}
\label{inequal-eigenvalues}
\lambda_{N, k}^\ve \le \lambda_k^\ve \le \lambda_{D, k}^\ve, \quad k = 1,2, \cdots.
\end{equation}
It is well-known (see \cite{Van-81}) that $\lambda_{N, k}^\ve = O(1)$ and $\lambda_{D, k}^\ve=O(\ve^{-2})$, $\ve \to 0$. Lemma~\ref{lemma-est-lambda^eps} specifies estimate \eqref{inequal-eigenvalues} for the first eigenvalue $\lambda_1^\ve$.
\end{remark}

\subsection{Change of unknowns. Rescaled problem}
For brevity, we denote
$$
\varkappa(x) = \frac{|\Sigma^0|_{d-1}}{|Y|_d}\, q(x), \quad
Q = \frac{1}{2}\, \frac{|\Sigma^0|_{d-1}}{|Y|_d}\,H(q),
$$
where $H(q)$ is the Hessian matrix of $q$ at $x=0$.

Note that Lemma~\ref{lemma-est-lambda^eps} suggests to study the asymptotics of $(\lambda_k^\ve - \ve^{-1}\,\varkappa(0))$, rather than of $\lambda_k^\ve$ itself. On the other hand, when deriving the upper bound in Lemma~\ref{lemma-est-lambda^eps}, we used the test function $v(x/\ve^{1/4})$, which allowed us to get the "optimal" estimate. Bearing in mind these two ideas, we first subtract $\ve^{-1}\,\varkappa(0)\,u^\ve(x)$ from both sides of the equation in \eqref{or-prob}, and then make the change of variables $z = \ve^{-1/4} x$ in \eqref{or-prob}. Then, the rescaled problem is stated in the domain
$$
\widetilde{\Omega_\ve} = \ve^{-1/4}\, \Omega_\ve,\qquad
\widetilde{\Sigma_\ve}=\ve^{-1/4}\, \Sigma_\ve,
$$
and takes the form
\begin{equation}
\label{rescaled-prob}
\left\{
\begin{array}{lcr}
\displaystyle
- {\rm div}(a^\ve(z) \nabla v^\ve(z)) -\frac{\varkappa(0)}{\sqrt{\ve}}\, v^\ve = \mu^\ve \, v^\ve(x), \qquad \hfill z \in \widetilde{\Omega_\ve},
\\[3mm]
a^\ve(z) \nabla v^\ve(z)\cdot n = - \ve^{1/4}\,q(\ve^{1/4} z)\, v^\ve(z), \qquad \hfill z \in \widetilde{\Sigma_\ve},
\\[3mm]
v^\ve(z)=0, \qquad \hfill z \in \ve^{-1/4}\partial \Omega.
\end{array}
\right.
\end{equation}
Here
\begin{equation}
\label{change-var}
v^\ve(z) = u^\ve(\ve^{1/4} z), \quad
a^\ve(z) = a\big(\frac{z}{\ve^{3/4}}\big), \quad
\mu^\ve = \sqrt{\ve}\, \big(\lambda^\ve - \frac{\varkappa(0)}{\ve}\big).
\end{equation}
The weak formulation of problem \eqref{prop-spectr-rescaled} reads: find $(\mu^\ve, v^\ve) \in \mathbb{R}\times H_0^1(\widetilde{\Omega_\ve}, \ve^{-1/4} \partial \Omega)$, $v^\ve \neq 0$, such that
\begin{equation}
\label{weak-rescal-prob}
W^\ve(v^\ve, w) = \mu^\ve (v^\ve, w)_{L^2(\widetilde{\Omega_\ve})}, \quad \forall w \in H_0^1(\widetilde{\Omega_\ve}, \ve^{-1/4} \partial \Omega).
\end{equation}
Here the bilinear form $W^\ve(u,v)$ is given by
\begin{equation}
\label{bilin-W^eps}
W^\ve(u,v)=
\int \limits_{\widetilde{\Omega_\ve}} a^\ve \nabla u \cdot \nabla v\, dz
- \frac{\varkappa(0)}{\sqrt{\ve}}\, \int \limits_{\widetilde{\Omega_\ve}} u\, v\, dz
+ \ve^{1/4}\, \int \limits_{\widetilde{\Sigma_\ve}} q(\ve^{1/4}z)\, u\, v\, d\sigma_z.
\end{equation}
\begin{remark}
\label{remark-extension}
\textbf{About the extension operator}
For all sufficiently small $\ve$, there exists an extension operator
$$
P^\ve: H_0^1(\widetilde{\Omega}_\ve, \ve^{-1/4}\partial \Omega) \to H_0^1(\ve^{-1/4} \Omega)
$$
such that
$$
\|P^\ve v\|_{L^2(\ve^{-1/4} \Omega)} \le C\, \|v\|_{L^2(\widetilde{\Omega_\ve})},\qquad
\|\nabla(P^\ve v)\|_{L^2(\ve^{-1/4} \Omega)} \le C\, \|\nabla v\|_{L^2(\widetilde{\Omega_\ve})},
$$
where $C$ is a constant independent of $\ve$.

Moreover, the obtained extended function (for which we keep the same notation) can be extended by zero to the whole $\mathbb{R}^d$, outward the boundary $\ve^{-1/4}\partial \Omega$.
\end{remark}

\begin{proposition}
\label{prop-spectr-rescaled}
The spectrum of problem \eqref{weak-rescal-prob} is real, discrete and consists of a countable set of points
$$
0 < \mu_1^\ve < \mu_2^\ve \le \cdots \le \mu_j^\ve \le \cdots \to +\infty.
$$
The corresponding eigenfunctions can be normalized by
\begin{equation}
\label{norm-cond-v^eps}
W^\ve(v_i^\ve, v_j^\ve) = \delta_{ij},
\end{equation}
with $W^\ve(u,v)$ defined by \eqref{bilin-W^eps}.
\end{proposition}
\begin{proof}
For any fixed $\ve >0$, the bilinear form $W^\ve(\cdot, \cdot)$ defines an equivalent scalar product in $H_0^1(\widetilde{\Omega_\ve}, \ve^{-1/4}\partial \Omega)$. For brevity, we denote
\be
\label{space}
H_{0,W}^1(\widetilde{\Omega_\ve}) = \{w \in H_0^1(\widetilde{\Omega_\ve}, \ve^{-1/4}\partial \Omega):\,\, \|v\|_{\ve,W}^2=W^\ve(w,w)<\infty\}.
\ee
Let $G^\ve: L^2(\widetilde{\Omega_\ve}) \to H_{0,W}^1(\widetilde{\Omega_\ve})$ be the operator defined as follows:
\be
\label{operator-G^eps}
W^\ve(G^\ve f, w) = (f, w)_{L^2(\widetilde{\Omega_\ve})}, \quad w \in H_0^1(\widetilde{\Omega_\ve}, \ve^{-1/4} \partial \Omega).
\ee
Obviously, $G^\ve$ is a positive, bounded (uniformly in $\ve$), self-adjoint operator. Since $H_{0,W}^1(\widetilde{\Omega_\ve})$, for each fixed $\ve$, is compactly embedded into $L^2(\widetilde{\Omega})$, then $G^\ve$ is compact as an operator from $L^2(\widetilde{\Omega})$ ($H_{0,W}^1(\widetilde{\Omega_\ve})$) into itself.

Thus, the spectrum $\sigma(G^\ve)$ is a countable set of points in $\mathbb{R}$ which does not have any accumulation points except for zero. Every nonzero eigenvalue has finite multiplicity. To complete the proof of the proposition, it is left to notice that in terms of the operator $G^\ve$ the eigenvalue problem \eqref{rescaled-prob} takes the form
$$
G^\ve\, v^\ve = \frac{1}{\mu^\ve}\, v^\ve.
$$
\end{proof}

We proceed with auxiliary technical results that will be useful in the sequel.

Define the following norms in $H^1(\widetilde{\Omega_\ve})$:
$$
\|v\|_{\ve, W}^2 = \int \limits_{\widetilde{\Omega_\ve}} a^\ve \nabla v\cdot \nabla v\, dz
- \frac{\varkappa(0)}{\sqrt{\ve}}\, \int \limits_{\widetilde{\Omega_\ve}} |v|^2\, dz
+ \ve^{1/4}\, \int \limits_{\widetilde{\Sigma_\ve}} q(\ve^{1/4}z)\, |v|^2\, d\sigma_z;
$$
\begin{equation}\label{au_norms}
\|v\|_{\ve,\varkappa}^2= \int \limits_{\widetilde{\Omega_\ve}} a^\ve \nabla v \cdot \nabla v\, dz
+ \frac{1}{\sqrt{\ve}}\int \limits_{\widetilde{\Omega_\ve}} (\varkappa(\ve^{1/4}z) - \varkappa(0))\,|v|^2 dz;
\end{equation}
$$
\|v\|_{\ve, Q}^2 = \int \limits_{\widetilde{\Omega_\ve}} a^\ve \nabla v\cdot \nabla v\, dz
+ \int \limits_{\widetilde{\Omega_\ve}} (z^T Q z)\,|v|^2 dz.
$$
\begin{lemma}
\label{lemma-norm-equiv}
The norms  $\|\cdot\|_{\ve, W}$, \ $\|\cdot\|_{\ve,\varkappa}$ and  $\|\cdot\|_{\ve, Q}$  are equivalent. Morover,
\begin{equation}
\label{equiv-norms}
\begin{array}{l}
\displaystyle
C_1 \, \|v\|_{\ve,\varkappa}^2 \le \|v\|_{\ve, W}^2 \le C_2 \, \|v\|_{\ve,\varkappa}^2;
\\[3mm]
\displaystyle
C_3 \, \|v\|_{\ve,\varkappa}^2 \le \|v\|_{\ve, Q}^2 \le C_4 \, \|v\|_{\ve,\varkappa}^2,
\end{array}
\end{equation}
with constants $C_1, C_2, C_3$ and $C_4$ that do not depend on $\ve$.
\end{lemma}
\begin{proof}
Indeed, by Lemma~\ref{lemma-1-bis} and by the Poincar\'{e} inequality,
$$
\begin{array}{l}
\displaystyle
|\|v\|_{\ve,W}^2 - \|v\|_{\ve,\varkappa}^2 | \le C\,\ve^{1/4}\, \|v\|_{L^2(\widetilde{\Omega_\ve})}\,\|\nabla v\|_{L^2(\widetilde{\Omega_\ve})}
\le C_1\, \ve^{1/4}\, \|v\|_{\ve,\varkappa}^2
\end{array}
$$
and, thus, the first inequality in \eqref{equiv-norms} holds for sufficiently small $\ve$.

The second inequality follows easily from the hypothesis (\textbf{H4}) and Lemma~\ref{lemma-2}.
\end{proof}

\begin{remark}
\label{remark-equiv-norm}
If $v \in H^1(\widetilde{\Omega}_\ve)$ decays exponentially, namely,
$$
\|v\|_{L^2(\mathbb{R}^d\setminus B_R(0))} \le M\, e^{- \gamma_0\, R},
$$
for some constant $M$, then the norms defined in Lemma~\ref{lemma-norm-equiv} are asymptotically close. In particular, the following estimate holds:
$$
|\|v\|_{\ve, W}^2 - \|v\|_{\ve, Q}^2| \le C\, \ve^{1/4}
$$
with the constant $C=C(M, \gamma_0)$ independent of $\ve$.
\end{remark}

\begin{lemma}
\label{lemma-est-mu^eps}
Let $\mu_1^\ve$ be the first eigenvalue of the spectral problem \eqref{rescaled-prob}. Then there exist two positive constants $C_1$ and $C_2$ such that
$$
C_1 \le \mu_1^\ve \le C_2.
$$
\end{lemma}
\begin{proof}
The upper bound follows from \eqref{change-var} and Lemma~\ref{lemma-est-lambda^eps}.
The lower bound is the consequence of the boundedness of the operator $G^\ve$ (see the proof of Proposition~\ref{prop-spectr-rescaled}).
\end{proof}


\subsubsection{Formal asymptotic expansion for the rescaled problem}
Following the classical asymptotic expansion method and bearing in mind Lemma~\ref{lemma-est-mu^eps}, we seek for a solution of problem \eqref{rescaled-prob} in the form of asymptotic series
\begin{equation}
\label{ansatze}
\begin{array}{l}
\displaystyle
\mu^\ve = \mu + \ve^{1/4} \mu_{1 \over 4} + \ve^{1/2} \mu_{1 \over 2} + \cdots,
\\[2mm]
\displaystyle
v^\ve = v(z) + \ve^{1/4}\, v_{1 \over 4}(z,\zeta) + \ve^{1/2}\, v_{1\over 2}(z,\zeta) + \ve^{3/4}\, v_{3\over 4}(z,\zeta) +\cdots, \,\, \zeta = \frac{z}{\ve^{3/4}},
\end{array}
\end{equation}
where the functions $v_{k\over 4}(z,\zeta)$ are $Y$-periodic in $\zeta$, \, $k=1,2,\dots$.

Substituting ans\"{a}tze \eqref{ansatze} into \eqref{rescaled-prob} and collecting the terms of order $\ve^{-5/4}$ and $\ve^{-1}$ in the equation, and of order $\ve^{-1/2}$, $\ve^{-1/4}$ in the boundary condition, we see that the functions $v_{1 \over 4}$ and $v_{1\over 2}$ do not depend on $\zeta$. Then, collecting the terms of order $\ve^{-3/4}$, we obtain that
$$
v_{3\over 4}(z,\zeta) = N_k(\zeta)\,\partial_k v(z) + w_3(z),
$$
where the vector function $N(\zeta)$ solves the problem
\begin{equation}
\label{eq-N}
\left\{
\begin{array}{lcr}
- {\rm div}_{\zeta} (a(\zeta) \nabla_{\zeta} N_k(\zeta)) = {\rm div}_{\zeta} a_{\cdot k}(\zeta), \quad k = 1,...,d, \quad \zeta \in Y,
\\[2mm]
a\nabla_\zeta N_k \cdot n= - a_{i k} n_i, \quad \zeta \in \Sigma^0,
\\[2mm]
N_k(\zeta)\in H_{\#}^1(Y),
\end{array}
\right.
\end{equation}
The effective spectral problem comes out while collecting the terms of order $\ve^0$ and writing the compatibility condition for the resulting problem. It reads
\be
\label{eff-prob}
-{\rm div}(a^\eff \nabla v) + (z^T Q z)\, v = \mu\, v, \quad
v \in L^2(\mathbb{R}^d),
\ee
where $a^\eff$ is given by
\be
\label{a^eff}
a_{ij}^\eff = {1\over {|Y|_d}}\, \int \limits_Y a_{ik}(y)(\delta_{kj} + \partial_k N_j)\, dy.
\ee
The effective problem describes the eigenvalues and eigenfunctions of $d$-dimensional harmonic oscillator. In $\mathbb{R}^1$ an explicit solution can be given in terms of Hermite polynomials. In the case under consideration we prove the following statement that characterizes the spectrum of problem \eqref{eff-prob}.
\begin{lemma}
\label{lemma-spectr-eff}
The spectrum of the effective problem \eqref{eff-prob} is real and discrete
$$
0 < \mu_1 < \mu_2 \le \cdots \le \mu_j \cdots \to +\infty.
$$
The corresponding eigenfunctions $v_j(z)$ can be normalized by
\begin{equation}
\label{norm-cond-v}
(v_i, v_j)_Q \equiv \int \limits_{\mathbb{R}^d} a^\eff \nabla v_i \cdot \nabla v_j\, dz
+\int \limits_{\mathbb{R}^d} (z^T Q z)\, v_i\, v_j\, dz = \delta_{ij}.
\end{equation}
\end{lemma}
We omit the proof of Lemma~\ref{lemma-spectr-eff} which is classical.

It is well known that the eigenfunctions of the harmonic oscillator operator have the form
\be
\label{harm-oscil-v}
v_j(z) = P_{j-1}(z)\, e^{- z^T\, R\, z}, \quad R= \frac{\sqrt{2}\, Q^{1/2}\, (a^\eff)^{-1/2}}{2},
\ee
where $P_k(z)$ is a polynomial of degree $k$.

To summarize, the formal asymptotic expansion for $v^\ve$ takes the form
$$
v(z) + \ve^{3/4}\, N\big({z \over \ve^{3/4}}\big)\cdot \nabla v(z),
$$
where $v$ is an eigenfunction of the limit spectral problem \eqref{eff-prob}, $N$ is a periodic vector function solving \eqref{eq-N}.

Notice that we can neglect the summands $v_{\frac{1}{4}}$ and $v_{\frac{1}{2}}$ since they do not depend on the fast variable $\zeta$, and thus, their $H^1$-norm is of order $\ve^{1/4}$.


\subsubsection{Justification}
Denote $J(j)=\min\{i\in \mathbb Z^+\,:\,\mu_i=\mu_j\}$, and let $\kappa_j$ be the multiplicity of the $j$th eigenvalue $\mu_j$  of the harmonic oscillator operator \eqref{eff-prob}.

The main goal of this section is to prove the following theorem.
\begin{theorem}
\label{Th-main}
Let hypotheses \textbf{(H1)-(H4)} be fulfilled. If $(\mu_p^\ve, v_p^\ve)$ stands for $p$th eigenpair of problem \ref{rescaled-prob}, then the following statements hold true:
\begin{enumerate}
\item
For each $j=1,2,\dots,$ there exist $\ve_j > 0$ and a constant $c_j$
such that the eigenvalue $\mu_j^\ve$ of problem \eqref{rescaled-prob}
satisfies the inequality
$$
|\mu_j^\ve - \mu_j| \leq c_j \, \ve^{1/4}, \quad \ve \in (0, \ve_j),
$$
where $\mu_j$ is an eigenvalue of the harmonic oscillator operator \eqref{eff-prob}.
\item
There exists a unitary $\kappa_j \times \kappa_j$ matrix $\beta^\ve$ such that
\begin{equation}
\label{13}
\bigg\|v_p^\ve - \sum \limits_{k = J(j)}^{J(j) + \kappa_j - 1} \beta_{pk}^\ve \, \widetilde{V}_k^\ve \bigg\|_{\ve,Q} \leq
C_j \, \ve^{1/4}, \quad p = J(j), \cdots, J(j)+ \kappa_j - 1,
\end{equation}
where
\begin{equation}
\label{anzats-without-cut-off}
\widetilde{V}_k^\ve = v_k(z) + \ve^{3/4} \, N\big({z\over {\ve^{3/4}}}\big) \cdot \nabla v_k(z).
\end{equation}
Here the vector function $N(\zeta)$ solve problem \eqref{eq-N}; eigenfunctions $v_k(z)$ of the limit problem are defined in \eqref{eff-prob}; the norm $\|\cdot\|_{\ve,Q}$ is defined just before Lemma~\ref{lemma-norm-equiv}.

Moreover, almost eigenfunctions $\{\widetilde{V}_k^\ve\}$ satisfy the following orthogonality and normalization condition:
\begin{equation}
\label{orthog-cond-ansatz}
\Big|\int \limits_{\widetilde{\Omega}_\ve} a^\ve \nabla \widetilde{V}_k^\ve \cdot \nabla \widetilde{V}_m^\ve\, dz
+ \int \limits_{\widetilde{\Omega}_\ve} (z^T Q z)\, \widetilde{V}_k^\ve \,\widetilde{V}_m^\ve\, dz - \delta_{km}\Big| \le C\, \ve^{1/4}.
\end{equation}
\end{enumerate}
\end{theorem}
\begin{proof}
The justification procedure will rely on Vishik's lemma about "almost eigenvalues and eigenfunctions" (see, for example, \cite{VishLust} and \cite{Naz2002}, p. 319, Lemma~1.5). For the reader's convenience, we formulate the mentioned result.
\begin{lemma}
\label{lemma Vishik}
Given a self-adjoint operator $\mathcal{K}^\ve: \mathcal{H} \to \mathcal{H}$ with a discrete spectrum, let $\nu \in \mathbb{R}$ and $v \in \mathcal{H}$ be such that
$$
\|v\|_{\mathcal{H}} = 1, \quad
\delta \equiv \|\mathcal{K}^\ve \, v - \nu \, v \|_{\mathcal{H}} < |\nu|.
$$
Then there exists an eigenvalue $\mu_l^\ve$ of the operator $\mathcal{K}^\ve$ such that
$$
|\mu_l^\ve - \nu| \leq \delta.
$$
Moreover, for any $\delta_1 \in (\delta, |\nu|)$ there exist $\{a_j^\ve\} \in \mathbb{R}$ such that
$$
\|v - \sum a_j^\ve u_j^\ve \|_{\mathcal{H}} \leq 2 \frac{\delta}{\delta_1},
$$
where the sum is taken over the eigenvalues of the operator $\mathcal{K}^\ve$ on the segment $[\nu - \delta_1, \nu + \delta_1]$, and $\{u_j^\ve\}$ are the corresponding eigenfunctions. The coefficients $a_j^\ve$ are normalized so that $\sum |a_j^\ve|^2 = 1$.
\end{lemma}

Let $\mu_j$ be an eigenvalue of the effective problem \eqref{eff-prob} of multiplicity $\varkappa_j$ that is $\mu_j=\mu_{j+1}=\dots=\mu_{j+\varkappa_j-1}$, and  $\{v_p(z)\}$, $p=j, \cdots, j+{\varkappa}_j - 1$, be the eigenfunctions corresponding to $\mu_j$. Denote
\be
\label{appr-sol-V^eps}
V_p^\ve(z) = v_p(z)\,\chi_\ve(z) + \ve^{3/4}\, \chi_\ve(z)\, N\big({z \over \ve^{3/4}}\big)\cdot \nabla v_p(z),
\ee
where $v_p$ is the $p$th eigenfunction of the limit spectral problem \eqref{eff-prob}, $N$ is a solution of \eqref{eq-N}; $\chi_\ve(z)$ is a cut-off which is equal to $1$ if $|z|<\frac{\ve^{-1/4}}{3}\,{\rm dist}(0, \partial \Omega)$, equal to $0$ if $|z|>\frac{\ve^{-1/4}}{2}\,{\rm dist}(0, \partial \Omega)$, and is such that
\begin{equation}
\label{chi_eps(x)}
0 \leq \chi_\ve(x) \leq 1, \qquad
|\nabla \chi_\ve| \leq C \ve^{1/4}.
\end{equation}
We apply Lemma~\ref{lemma Vishik} to the operator $G^\ve: H_{0,W}^1(\widetilde{\Omega_\ve}) \to  H_{0,W}^1(\widetilde{\Omega_\ve})$ constructed in Proposition~\ref{prop-spectr-rescaled} (see \eqref{operator-G^eps}). The normalized functions $\mathcal{V}_p^\ve \equiv V_p^\ve/ \|V_p^\ve\|_{\ve,W}$ and the numbers $\mu_j$ will play the roles of $v \in \mathcal{H}$ and $\nu \in \mathbb{R}$ in Lemma~\ref{lemma Vishik}. Notice that $v_j$  need not be equal to zero on the boundary $\ve^{-1/4}\partial \Omega$; the cut-off function has been introduced in order to make approximate solution \eqref{appr-sol-V^eps} belong to the space $H_{0,W}^1(\widetilde{\Omega_\ve})$ (see \eqref{space}).

\begin{lemma}
\label{lemma-orthonor-V^eps}
"Almost" eigenfunctions $V_p^\ve$ are almost orthonormal. Namely, the following inequalities hold:
\begin{equation}
\label{norm-cond-V_p^eps}
\begin{array}{l}
\displaystyle
|W^\ve({V}_p^\ve, {V}_q^\ve) - \delta_{pq}| \le C\, \ve^{1/4},
\\[3mm]
\displaystyle
|({V}_p^\ve, {V}_q^\ve)_{\ve,Q} - \delta_{pq}| \le C\, \ve^{1/4}.
\end{array}
\end{equation}
where $W^\ve(u,v)$ and $(\cdot, \cdot)_Q$ are defined by \eqref{bilin-W^eps} and \eqref{norm-cond-v}, respectively.
\end{lemma}
\begin{proof}
We calculate first the gradient of the function $V_p^\ve$.
$$
\nabla V_p^\ve = J_{1p}^\ve(z)\, \chi_\ve(z) +
\ve^{3/4}\,J_{2p}^\ve(z)+
J_{3p}^\ve(z)\, \nabla \chi_\ve(z),
$$
where
$$
\begin{array}{l}
\displaystyle
J_{1p}^\ve(z)
= \nabla v_p(z) +
\nabla_{\zeta}(N(\zeta)\cdot v_p(z)) \Big|_{\zeta = z/\ve^{3/4}};
\end{array}
$$
$$
J_{2p}^\ve(z)
= \chi_\ve(z)\, \nabla_z(N(\zeta)\cdot \nabla v_p(z));
$$
$$
J_{3p}^\ve(z)= v_p(z)+\ve^{3/4} N\big(\frac{z}{\ve^{3/4}}\big)\cdot \nabla v_p(z).
$$
One can show that
$$
\begin{array}{l}
\displaystyle
\bigg|W^\ve({V}_p^\ve, {V}_q^\ve) -
\int\limits_{\widetilde{\Omega}_\ve} a^\ve(z) (\chi_\ve(z))^2 \, J_{1p}^\ve \cdot J_{1q}^\ve\, dz
\\[3mm]
\displaystyle
+\frac{\varkappa(0)}{\sqrt{\ve}} \int\limits_{\widetilde{\Omega}_\ve} v_p(z)\, v_q(z)\, (\chi_\ve(z))^2\, dz
\\[3mm]
\displaystyle
- \ve^{1/4}\, \int \limits_{\Sigma_\ve} q(\ve^{1/4}) v_p(z)\,v_q(z)\, (\chi_\ve(z))^2 \, d\sigma_z\bigg| \le C\, \ve^{1/4}.
\end{array}
$$
On the other hand, using Lemma~\ref{lemma-mean-value}, exponential decay of the eigenfunctions $v_p(z)$ and the normalization condition \eqref{norm-cond-v}, we can prove that
$$
\begin{array}{l}
\displaystyle
\bigg|\int\limits_{\widetilde{\Omega}_\ve} a^\ve(z) (\chi_\ve(z))^2 \, J_{1p}^\ve \cdot J_{1q}^\ve\, dz
\\[3mm]
\displaystyle
-\frac{\varkappa(0)}{\sqrt{\ve}} \int\limits_{\widetilde{\Omega}_\ve} v_p(z)\, v_q(z)\, (\chi_\ve(z))^2\, dz
\\[3mm]
\displaystyle
+ \ve^{1/4}\, \int \limits_{\Sigma_\ve} q(\ve^{1/4}) v_p(z)\,v_q(z)\, (\chi_\ve(z))^2 \, d\sigma_z - \delta_{pq}\bigg| \le C\, \ve^{1/4}.
\end{array}
$$
Combining the last two estimates, we get
$$
\big|W^\ve({V}_p^\ve, {V}_q^\ve) - \delta_{pq}\big| \le C\, \ve^{1/4},
$$
which is the first estimate in \eqref{norm-cond-V_p^eps}.

The second estimate in \eqref{norm-cond-V_p^eps} follows from the first one and Remark~\ref{remark-equiv-norm}.

\end{proof}

\begin{lemma}
\label{lemma-est-forVishik}
Let $\mathcal{V}_p^\ve \equiv V_p^\ve/ \|V_p^\ve\|_{\ve,W}$ with $V_p^\ve$ defined by \eqref{appr-sol-V^eps}. Then the following estimate holds:
\begin{equation}\label{est_l_38}
\|G^\ve \mathcal{V}_p^\ve - (\mu_j)^{-1} \, \mathcal{V}_p^\ve\|_{\ve, W} \le C_p\, \ve^{1/4}, \quad p = i, \cdots, i + \kappa_j -1.
\end{equation}
\end{lemma}
\begin{proof}
Simple transformations result in the following relations:
$$
\|G^\ve \mathcal{V}_p^\ve - (\mu_j)^{-1} \, \mathcal{V}_p^\ve\|_{\ve, W}
= \|V_p^\ve\|_{\ve, W}^{-1} \, \sup \limits_{\|w\|_{\ve,W}=1} \,\, W^\ve\big(G^\ve V_p^\ve - (\mu_j)^{-1} \,V_p^\ve, w\big)
$$
By \eqref{operator-G^eps},
$$
\ba{l}
\disp
\|G^\ve \mathcal{V}_p^\ve - (\mu_j)^{-1} \, \mathcal{V}_p^\ve\|_{\ve, W}
= \frac{1}{\mu_p}\, \|V_p^\ve\|_{\ve, W}^{-1} \, \sup \limits_{\|w\|_{\ve,W}=1}
\big\{\mu_p (V_p^\ve, w)_{L^2(\widetilde{\Omega_\ve})}
\cr\cr
\disp
-\int \limits_{\widetilde{\Omega_\ve}} a^\ve \nabla V_p^\ve\cdot \nabla w\, dz
+ \frac{\varkappa(0)}{\sqrt{\ve}}\, \int \limits_{\widetilde{\Omega_\ve}}V_p^\ve\, w\, dz
- \ve^{1/4}\, \int \limits_{\widetilde{\Sigma_\ve}} q(\ve^{1/4}z)\, V_p^\ve\, w\, d\sigma_z\big\}
\cr\cr
\disp
= \frac{1}{\mu_p}\, \|V_p^\ve\|_{\ve, W}^{-1} \, \sup \limits_{\|w\|_{\ve,W}=1}
\big\{I_1^\ve + I_2^\ve + \ve^{3/4}\,I_3^\ve\big\}.
\ea
$$
Here
$$
\ba{l}
\disp
I_1^\ve = \mu_p\, \int \limits_{\widetilde{\Omega_\ve}} \, \chi_\ve(z)\,v_p(z)\,w(z) \,dz
- \frac{1}{\sqrt{\ve}} \,\int \limits_{\widetilde{\Omega_\ve}} (\varkappa(\ve^{1/4}z) - \varkappa(0))\,\chi_\ve(z)\, v_p(z)\, w(z)\, dz
\cr\cr
\disp
- \int \limits_{\widetilde{\Omega_\ve}} a(\zeta)\Big(\nabla v_p(z) + \nabla_\zeta(N(\zeta)\cdot \nabla v_p(z))\big)\cdot \nabla w \, \chi_\ve(z) \big|_{\zeta = z/\ve^{3/4}}\, dz;
\ea
$$
\medskip
$$
I_2^\ve = \frac{1}{\sqrt{\ve}} \,\int \limits_{\widetilde{\Omega_\ve}} \varkappa(\ve^{1/4}z)\, v_p(z)\,\chi_\ve(z)\, w(z)\, dz
- \ve^{1/4}\, \int \limits_{\widetilde{\Sigma_\ve}} q(\ve^{1/4}z)\,v_p(z)\,\chi_\ve(z)\, w(z)\, d\sigma;
$$
$$
\ba{l}
\disp
I_3^\ve = \mu_p\, \int \limits_{\widetilde{\Omega_\ve}} \chi_\ve(z)\,N\big(\frac{z}{\ve^{3/4}}\big)\cdot \nabla v_p(z)\,w(z) \,dz
\cr\cr
\disp
-\int \limits_{\widetilde{\Omega_\ve}} a(\zeta)\nabla \chi_\ve(z)\cdot \nabla w\, v_p(z)\big|_{\zeta = z/\ve^{3/4}}\, dz
\cr\cr
\disp
- \int \limits_{\widetilde{\Omega_\ve}} a(\zeta)\nabla_z(\chi_\ve(z)\, N(\zeta)\cdot \nabla v_p(z))\cdot \nabla w\,\big|_{\zeta = z/\ve^{3/4}}\, dz
\cr\cr
\disp
+ \ve^{1/4}\, \varkappa(0)\, \int \limits_{\widetilde{\Omega_\ve}} \chi_\ve(z)\, N\big(\frac{z}{\ve^{3/4}}\big)\cdot \nabla v_p(z)\, w(z)\, dz
\cr\cr
\disp
- \ve^{1/4}\, \int \limits_{\widetilde{\Sigma_\ve}} q(\ve^{1/4}z)\, \chi_\ve(z)\, N\big(\frac{z}{\ve^{3/4}}\big)\cdot \nabla v_p(z)\, w\, d\sigma.
\ea
$$
Integrating by parts in the last integral in $I_1^\ve$, taking into account \textbf{(H4)}, \eqref{eq-N} and \eqref{chi_eps(x)}, we obtain
$$
\ba{l}
\disp
I_1^\ve = \mu_p\, \int \limits_{\widetilde{\Omega_\ve}} \, \chi_\ve(z)\,v_p(z)\,w(z) \,dz
- \int \limits_{\widetilde{\Omega_\ve}} (z^T Q z)\,\chi_\ve(z)\, v_p(z)\, w(z)\, dz
\cr\cr
\disp
+ \int \limits_{\widetilde{\Omega_\ve}} {\rm div}_z\big(a(\zeta)(I + \nabla_\zeta N(\zeta))\nabla v_p(z)\big)\big|_{\zeta = z/\ve^{3/4}}\, w(z) \, \chi_\ve(z) \, dz + O(\ve^{1/4}), \quad \ve \to 0.
\ea
$$
Here we have also used Lemma~\ref{lemma-norm-equiv} and the fact that $\|w\|_{\ve, W}=1$.

Bearing in mind the definition of the effective diffusion \eqref{a^eff} and \eqref{eff-prob}, by virtue of Lemma~\ref{lemma-mean-value}, one has
\be
\label{7}
|I_1^\ve|\le C\, \ve^{1/4}\, \|w\|_{H^1(\mathbb{R}^d)}.
\ee
By Lemma~\ref{lemma-1-bis},
\be
\label{8}
|I_2^\ve| \le C_2\, \ve^{1/4}\,\|v_p\|_{H^1(\mathbb{R}^d)}\, \|w\|_{H^1(\mathbb{R}^d)}.
\ee
Using the boundedness of $a_{ij}$ and the regularity properties of $N, v_p, \chi_\ve$, one can show that
\be
\label{9}
|I_3^\ve| \le C_3\, \|\nabla v_p\|_{L^2(\mathbb{R}^d)}\, \|w\|_{H^1(\mathbb{R}^d)}.
\ee
Using Lemma~\ref{lemma-orthonor-V^eps} we see that for small enough $\ve$,
\be
\label{11}
\|V_p^\ve\|_{\ve,W}^2 \ge \frac{1}{2}.
\ee
Finally, combining \eqref{7}-\eqref{11} we obtain the desired estimate (\ref{est_l_38}).
Lemma~\ref{lemma-est-forVishik} is proved.
\end{proof}
By Lemma~\ref{lemma Vishik}, in view of the estimate obtained in Lemma~\ref{lemma-est-forVishik}, for any eigenvalue $\mu_j$ of the effective problem \eqref{eff-prob} there exists an eigenvalue of the original problem such that
\be
\label{est-eigenv-1}
|\mu_{q(j)}^\ve - \mu_j|\le C_j\, \ve^{1/4},
\ee
where $q(j)$ might depend on $\ve$.

Moreover, letting $\delta_1$ in the statement of Lemma~\ref{lemma Vishik} be equal to $\Theta_j \, \ve^{1/4}$ (the constant $\Theta_j$ will be chosen below), we conclude that there exists a $K_j(\ve) \times {\kappa}_j$ constant matrix $\alpha^\ve$ such that
\begin{equation}
\label{14}
\bigg\|\mathcal{V}_p^\ve - \sum \limits_{k = J_j}^{J_j + K_J(\ve) - 1} \alpha_{kp}^\ve \, v_k^\ve \bigg\|_{\ve, W} \leq
2 {C\ve^{1/4} \over \delta_1} \leq C_j (\Theta_j)^{-1}\,  \quad p = j, \cdots, j+ {\kappa}_j - 1,
\end{equation}
here $\mu_k^\ve$, $k=J_j(\ve),\dots,J_j(\ve)+K_j(\ve)-1$, are all the  eigenvalues of operator $(G^\ve)^{-1}$ which satisfy the estimate
\be
\label{12}
|\mu_k^\ve - \mu_j| \leq \Theta_j \, \ve^{1/4}.
\ee
Since the eigenvalues $\mu_j$ do not depend on $\ve$, one can choose  constants $\ve_j>0$ so that the intervals $(\mu_j -\Theta_{j}\ve^{1/4}\,,\, \mu_j+\Theta_{j}\ve^{1/4})$ and $(\mu_i -\Theta_{i}\ve^{1/4}\,,\, \mu_i+\Theta_{i}\ve^{1/4})$  do not intersect
if $\mu_j\not=\mu_i$ and $\ve<\min(\ve_j,\ve_i)$.  Then the sets of eigenvalues $\{\mu_k^\ve\}$ related to different $\mu_j$ in (\ref{12}) do not intersect for sufficiently small $\ve$.

In the following statement we prove that $K_J(\ve) \geq {\kappa}_j$.
\begin{lemma}
\label{lemma-matrix-alpha^eps}
The columns of the matrix $\alpha^\ve$, that is the vectors $\{\alpha_{\cdot p}^\ve\}_{p = J(j)}^{J(j) + {\kappa}_j - 1}$ of length $K_J(\ve)$ are linearly independent. As a consequence, $K_J(\ve) \geq {\kappa}_j$.
\end{lemma}
\begin{proof}
A simple transformation gives
$$
\begin{array}{c}
W^\ve(\mathcal{V}_p^\ve, \mathcal{V}_q^\ve) =
W^\ve\Big(\mathcal{V}_p^\ve - \sum \limits_{k=J_j}^{J_j+ K_J(\ve) - 1} \alpha_{kp}^\ve v_k^\ve \, , \, \mathcal{V}_p^\ve\Big)+
\\[5mm]
+ W^\ve\Big( \sum \limits_{k=J_j}^{J_j+ K_J(\ve) - 1} \alpha_{kp}^\ve v_k^{\ve,\pm}  \, , \, \mathcal{V}_q^\ve - \sum \limits_{k=J_j}^{J_j+ K_J(\ve) - 1} \alpha_{kq}^\ve v_k^\ve\Big) +
\sum \limits_{k=J_j}^{J_j+ K_J(\ve) - 1} \alpha_{kp}^\ve \, \alpha_{kq}^\ve.
\end{array}
$$
Taking estimates \eqref{norm-cond-V_p^eps} and \eqref{14} into account, we obtain
$$
\Big| \sum \limits_{k=J_j}^{J_j+ K_J(\ve) - 1} \alpha_{kp}^\ve \, \alpha_{kq}^\ve -
\delta_{p,q} \Big| \leq C \, \Theta_j^{-1},\quad
p,q = J(j), \cdots, J(j)+ {\kappa}_j - 1,
$$
and, in other words,
\begin{equation}
\label{matrix-alpha^eps}
\big| (\alpha_{\cdot p}^\ve)^T \, \alpha_{\cdot q}^\ve - \delta_{p,q}\big| \leq C \, \Theta_j^{-1}, \quad
p,q = J(j), \cdots, J(j)+ {\kappa}_j - 1,
\end{equation}
where $\alpha_{\cdot p}^\ve$ denotes a $p$th column in the matrix $\alpha^\ve$. The last inequality means that the vectors $\{\alpha_{\cdot p}^\ve\}_{p = J(j)}^{J(j) + {\kappa}_j - 1}$ are asymptotically orthonormal, as $\Theta_j$ grows to infinity. This property implies the linear independence of the vectors $\{\alpha_{\cdot p}^\ve\}$ for sufficiently large $\Theta_j$. Indeed, assume that $\{\alpha_{\cdot p}^\ve\}_{p = J(j)}^{J(j) + {\kappa}_j - 1}$ are not linearly independent. Then there exist constants $c_{J(j)}, \cdots, c_{J(j) + {\kappa}_j -1}$ such that
$$
\sum \limits_{k=J(j)}^{J(j) + {\kappa}_j -1} c_k \, \alpha_{\cdot k}^\ve = 0.
$$
Without loss of generality we assume that $c_{J(j)} =1\ge \max_{k} |c_k|$. Then
$$
\alpha_{\cdot,J(j)}^\ve + \sum \limits_{k>J(j)} c_k \, \alpha_{\cdot k}^\ve = 0.
$$
Multiplying the last equality by $\alpha_{\cdot,J(j)}^\ve$ and using \eqref{matrix-alpha^eps} we obtain the inequality
$$
\big| (\alpha_{\cdot,J(j)}^\ve)^T \, \alpha_{\cdot,J(j)}^\ve \big| \leq C_j \, \Theta_j^{-1},
$$
that contradicts \eqref{matrix-alpha^eps} if $\Theta_j^{-1}$ is sufficiently small. Thus, the vectors $\{\alpha_{\cdot p}^\ve\}_{p = J(j)}^{J(j) + {\kappa}_j - 1}$ of length $K_J(\ve)$ are linearly independent. Obviously, it is possible only in the case $K_J(\ve) \geq {\kappa}_j$.
\end{proof}

\begin{lemma}
\label{lemma-bound-mu^eps}
For any $q$, $0 < m \le \mu_q^\ve \le M_q$.
\end{lemma}
\begin{proof}
The estimate from below is the immediate consequence of the boundedness of the operator $G^\ve$ constructed in Proposition~\ref{prop-spectr-rescaled}.

To obtain an upper bound for $\mu_q^\ve$, we recall estimate \eqref{est-eigenv-1}. For any $j$, there exists an eigenvalue of problem \eqref{rescaled-prob} converging to the $j$th eigenvalue of the effective problem. Namely, the estimate holds
$$
|\mu_{q_\ve(j)}^\ve - \mu_{J(j)}| \le C_j\, \ve^{1/4},
$$
where $J(j)=\min\{i\in \mathbb Z^+\,:\,\mu_i=\mu_j\}$. Obviously, $q_\ve(j) \ge J(j)$. Thus,
$$
\mu_{J(j)}^\ve \le \mu_{q_\ve(j)}^\ve \le \mu_{J(j)} + C_j\, \ve^{1/4}
$$
that implies the desired bound.
\end{proof}
Our next goal is to prove that any accumulation point of the sequence $\mu_q^\ve$, as $\ve \to 0$, is an eigenvalue of \eqref{eff-prob}.
\begin{lemma}
\label{lemma-3}
If, up to a subsequence, $\mu_q^\ve \to \mu^\ast$, as $\ve \to 0$, then $\mu^\ast$ is an eigenvalue of the effective spectral problem \eqref{eff-prob}. \end{lemma}
\begin{proof}
Since $\mu_q^\ve$ is bounded, then
$$
\|v_q^\ve\|_{\ve, W}\le C_q
$$
with $\|\cdot\|_{\ve, W}$ defined in (\ref{au_norms}).
In view of Lemmata~\ref{lemma-norm-equiv} and \ref{lemma-compact}, the eigenfunction $v_q^\ve$ (extended to the whole $\mathbb{R}^d$) converges weakly in $H^1(\mathbb{R}^d)$ and strongly in $L^2(\mathbb{R}^d)$ to some function $v^\ast$. To prove that $(\mu^\ast, v^\ast)$ is an eigenpair of the effective problem \eqref{eff-prob}, we pass to the limit in the integral identity \eqref{weak-rescal-prob}. Using standard two-scale convergence arguments we obtain
$$
\int \limits_{\mathbb{R}^d} a^\eff\nabla v^\ast \cdot \nabla w\, dz +
\int \limits_{\mathbb{R}^d} (z^T Q z)\, v^\ast\, w\, dz = \mu^\ast \,
\int\limits_{\mathbb{R}^d} v^\ast\, w\, dz, \quad w \in H^1(\mathbb{R}^d).
$$
The last equality is the weak formulation of \eqref{eff-prob}.
Since $\mu_q^\ve\to\mu^\ast$, as $\ve\to0$, then considering (\ref{weak-rescal-prob}) and (\ref{norm-cond-v^eps}) we conclude that $\lim\limits_{\ve\to0} \|v_q^\ve\|^2_{L^2(\widetilde\Omega_\ve)}=\mu^\ast$.
Using the strong convergence of $v_q^\ve$ in $L^2(\mathbb{R}^d)$, we see that $\|v^\ast\|^2_{L^2(\mathbb{R}^d)}\geq\mu^\ast\bigg.$.
By Lemma \ref{lemma-est-mu^eps} we have $\mu^\ast>0$. Therefore,
$v^\ast\neq 0$.
This completes the proof.
\end{proof}
\begin{lemma}
\label{lemma-multiplicity}
Let $\mu_j$ be the $j$th eigenvalue of problem \eqref{eff-prob} of multiplicity $\kappa_j$, that is $\mu_j=\mu_{j+1}=\cdots=\mu_{j+\kappa_j-1}$. Then there exist exactly $\kappa_j$ eiegnvalues of the original problem \eqref{or-prob} converging to it.
\end{lemma}
\begin{proof}
First, we prove that there are not more than $\kappa_j$ eigenvalues of problem \eqref{rescaled-prob} converging to $\mu_j$. Assume that there exist $\kappa_j+1$ eigenvalues $\mu_{J_k^\ve(j)}$ such that
$$
\mu_{J_k^\ve(j)} \to \mu_j, \quad k = 1, \cdots, \kappa_j + 1.
$$
By Lemma~\ref{lemma-3}, the corresponding eigenfunctions $v_{J_k^\ve(j)}$, extended to the whole $\mathbb{R}^d$, converge weakly in $H^1(\mathbb{R}^d)$ and strongly in $L^2(\mathbb{R}^d)$ to the eigenfunctions $v_k^\ast$ of the effective problem \eqref{eff-prob}, $k = 1, \cdots, \kappa_j+1$.
Passing to the limit in the normalization condition \eqref{norm-cond-v^eps} yields
$$
(v_i^\ast, v_k^\ast)_{L^2(\mathbb{R}^d)} = \frac{1}{|Y|_d}\mu_i\delta_{ik}, \quad i,k=1, \cdots, \kappa_j+1.
$$
Therefore, eigenfunctions $\{v_k^\ast\}_{k=1}^{\kappa_j + 1}$ corresponding to $\mu_j$ are linearly independent. Recalling that the multiplicity of $\mu_j$ is $\kappa_j$, we arrive at contradiction. Thus, there are not more than $\kappa_j$ eigenvalues of problem \eqref{rescaled-prob} converging to $\mu_j$.

On the other hand, by Lemma~\ref{lemma-matrix-alpha^eps}, there exist at least $\kappa_j$ eigenvalues of \eqref{rescaled-prob} converging to $\mu_j$ of multiplicity $\kappa_j$.
\ Lemma~\ref{lemma-multiplicity} is proved.
\end{proof}
Combining Lemmata~\ref{lemma-bound-mu^eps}--\ref{lemma-multiplicity} completes the proof of the first statement of Theorem~\ref{Th-main}.

We turn to the proof of the second statement in Theorem~\ref{Th-main}.

First of all, let us notice that the orthogonality and normalization condition \eqref{orthog-cond-ansatz} follows directly from Lemma~\ref{lemma-orthonor-V^eps} and the exponential decay of $v_k(z)$ as eigenfunctions of the harmonic oscillator.

In order to prove estimate \eqref{13}, we recall the estimate obtained in Lemma~\ref{lemma-est-forVishik} and apply the estimate in Lemma~\ref{lemma Vishik} with $\delta_1 = c_j$, $c_j$ being a sufficiently small constant. This estimate reads
$$
\bigg\|\mathcal{V}_p^\ve - \sum \limits_{\mu_k^\ve \in S(j,\ve)} \alpha_{kp}^\ve \, v_k^\ve \bigg\|_{\ve, W} \leq
2 {C\ve^{1/4} \over \delta_1} \leq C_j \ve^{1/4}\,  \quad p = j; \cdots, j+ {\kappa}_j - 1,
$$
where $S(j,\ve)$ is the set of eigenvalues  $\mu_k^\ve$ satisfying the estimate
$$
|\mu_k^\ve - \mu_j| \leq c_j;
$$
the constant matrix $\alpha^\ve$ is such that
\begin{equation}
\label{16}
\big| (\alpha_{\cdot p}^\ve)^T \, \alpha_{\cdot q}^\ve - \delta_{p,q}\big| \leq C_j \, \ve^{1/4}, \quad
p,q = J(j), \cdots, J(j)+ {\kappa}_j - 1,
\end{equation}
From the first statement of Theorem~\ref{Th-main} we deduce that the set $S(j,\ve)$ coincides with the set of eigenvalues $\{\mu_k^\ve\}_{J(j)}^{J(j) + \kappa_j -1}$, for sufficiently small $\ve$.
Therefore,
\begin{equation}
\label{15}
\bigg\|\mathcal{V}_p^\ve - \sum \limits_{k = J(j)}^{J(j)+\kappa_j-1} \alpha_{kp}^\ve \, v_k^\ve \bigg\|_{\ve, W}
\leq C_j \ve^{1/4}\,  \quad p = j; \cdots, j+ {\kappa}_j - 1,
\end{equation}
with a constant $\kappa_j \times \kappa_j$ matrix $\alpha^\ve$ which satisfies inequality \eqref{16}.

It remains to use the following simple statement.
\begin{lemma}
\label{lemma almost unitary matrix}
For any $n \times n$ matrix $A$ satisfying an equality
$$
\|A^T A - \mathbb{I}\|_{\mathcal{L}(\mathbb{R}^n,\mathbb{R}^n)} = \gamma \in (0,1),
$$
there exists a unitary matrix $B$  such that
$$
\|A B - \mathbb{I}\|_{\mathcal{L}(\mathbb{R}^n,\mathbb{R}^n)} \leq \gamma;
$$
here $\mathbb{I}$ is a unit matrix, and
$$
\|D\|_{\mathcal{L}(\mathbb{R}^n,\mathbb{R}^n)} = \sup \limits_{{\xi \in \mathbb{R}^n} \atop {\|\xi\| = 1}} \|D \xi\|.
$$
\end{lemma}
We omit the proof of this lemma which can be found in \cite{Naz2002}.
According to \eqref{16} and Lemma~\ref{lemma almost unitary matrix}, there exists a unitary $\kappa_j\times\kappa_j$ matrix $\beta^\ve$ such that
\begin{equation}
\label{alpha-to-beta}
\|\alpha^\ve \, \beta^\ve - \mathbb{I} \|_{\mathcal{L}(\mathbb{R}^{{\kappa}_j}, \mathbb{R}^{{\kappa}_j})} \leq C_j \ve^{1/4}.
\end{equation}
Taking into account Lemma~\ref{lemma-orthonor-V^eps}, estimates \eqref{15}, \eqref{alpha-to-beta}, one can show that
$$
\bigg\|v_p^\ve - \sum \limits_{k = J(j)}^{J(j) + \kappa_j - 1} \beta_{kp}^\ve \, V_k^\ve \bigg\|_{\ve, W} \leq
C_j \, \ve^{1/4}, \quad p = J(j), \cdots, J(j)+ \kappa_j - 1.
$$
Due to the exponential decay of the eigenfunctions $v_k(z)$ defined in \eqref{eff-prob}, one can replace $V_k^\ve$ defined by \eqref{appr-sol-V^eps} with \eqref{anzats-without-cut-off}.
Then, by Lemma~\ref{lemma-norm-equiv}, similar estimate holds for $\|\cdot\|_{\ve, Q}$ norm.
Theorem~\ref{Th-main} is proved.
\end{proof}
\bigskip

Bearing in mind the result obtained in Theorem~\ref{Th-main}, we formulate the main result of the present paper characterizing the asymptotic behaviour of eigenpairs $(\lambda_j^\ve, u_j^\ve)$ of problem \eqref{or-prob}.
\begin{theorem}
\label{Th-main-full}
Let conditions \textbf{(H1)-(H4)} be fulfilled. If $(\lambda_j^\ve, u_j^\ve)$ stands for the $j$th eigenpair of problem \eqref{or-prob}, then
for any $j$, the following representation takes place:
$$
\lambda_j^\ve = \frac{1}{\ve}\, \frac{|\Sigma^0 |_{d-1}}{|Y|_d}\, q(0) + \frac{\mu_j^\ve}{\sqrt{\ve}},
\quad u_j^\ve(x) = v_j^\ve \big(\frac{x}{\ve^{1/4}}\big),
$$
where the eigenpairs $(\mu_j^\ve, v_j^\ve(z))$ of problem \eqref{rescaled-prob} are such that
\begin{enumerate}
\item
For each $j=1,2,\dots,$ there exist $\ve_j > 0$ and a constant $c_j$
such that
$$
|\mu_j^\ve - \mu_j| \leq c_j \, \ve^{1/4}, \quad \ve \in (0, \ve_j),
$$
where $\mu_j$ is an eigenvalue of the harmonic oscillator operator \eqref{eff-prob}.
\item
Let $\mu_j$ be an eigenvalue of \eqref{eff-prob} of multiplicity $\kappa_j$, that is $\mu_j = \cdots = \mu_{j+\kappa_j-1}$.
Then, there exists a unitary $\kappa_j \times \kappa_j$ matrix $\beta^\ve$ such that
$$
\bigg\|v_p^\ve - \sum \limits_{k = J(j)}^{J(j) + \kappa_j - 1} \beta_{kp}^\ve \, \widetilde{V}_k^\ve \bigg\|_{\ve, Q} \leq
C_j \, \ve^{1/4}, \quad p = J(j), \cdots, J(j)+ \kappa_j - 1,
$$
where
$$
\widetilde{V}_k^\ve = v_k(z) + \ve^{3/4} \, N\big({z\over {\ve^{3/4}}}\big) \cdot \nabla v_k(z).
$$
Here the vector function $N(\zeta)$ solve problem \eqref{eq-N}; eigenfunctions $v_k(z)$ of the limit problem are defined in \eqref{eff-prob}; the norm $\|\cdot\|_{\ve,Q}$ is defined in (\ref{au_norms}).
\end{enumerate}
\end{theorem}

\section{Auxiliary results}
\begin{lemma}
\label{lemma-1}
For any $w^\ve(x)\in H_0^1(\Omega_\ve, \partial \Omega)$ the following estimate holds
$$
\Big|\frac{1}{\ve}\,\frac{|\Sigma^0|_{d-1}}{|Y|_d}\, \int \limits_{\Omega_\ve} |w^\ve|^2\, dx - \int \limits_{\Sigma_\ve} |w^\ve|^2\, d\sigma \Big|
\le C\, \|w^\ve\|_{L^2(\Omega_\ve)}\, \|\nabla w^\ve\|_{L^2(\Omega_\ve)}
$$
with a constant $C$ independent of $\ve$.
\end{lemma}
\begin{proof}
Introduce a $Y$-periodic vector function $\chi(y)$ as a solution of the following problem on the periodicity cell $Y$:
$$
\left\{
\begin{array}{l}
\displaystyle
-{\rm div}_y \chi = \frac{|\Sigma^0|_{d-1}}{|Y|_d}, \quad y \in Y,
\\[3mm]
(\chi, n)=-1, \quad y \in \Sigma^0.
\end{array}
\right.
$$
Notice that $\chi$ is a smooth function. Then
$$
-\ve\, {\rm div}_x \chi\big(\frac{x}{\ve}\big) = \frac{|\Sigma^0|_{d-1}}{|Y|_d}.
$$
Multiplying the last equality by $|w^\ve|^2$ and integrating by parts over $\Omega_\ve$ yields
$$
\frac{1}{\ve}\,\frac{|\Sigma^0|_{d-1}}{|Y|_d}\, \int \limits_{\Omega_\ve} |w^\ve|^2\, dx - \int \limits_{\Sigma_\ve} |w^\ve|^2\, d\sigma
= \int \limits_{\Omega_\ve} \big(\chi \big(\frac{x}{\ve}\big),\, \nabla |w^\ve|^2\big)\, dx
$$
that easily implies the desired estimate. Lemma is proved.
\end{proof}
\begin{lemma}
\label{lemma-1-bis}
Let $\widetilde{\Omega_\ve} = \ve^{-\alpha}\Omega_\ve, \,\, \widetilde{\Sigma_\ve} = \ve^{-\alpha}\Sigma_\ve$. Then, for $\psi(z)\in H_0^1(\widetilde{\Omega_\ve}, \ve^{-\alpha}\partial \Omega)$ and $\varphi \in C^1(\mathbb{R}^d)$, the following estimate holds
$$
\begin{array}{l}
\displaystyle
\Big|\frac{1}{\ve^{1-\alpha}}\,\frac{|\Sigma^0 |_{d-1}}{|Y|_d}\, \int \limits_{\widetilde{\Omega_\ve}} \varphi(\ve^{\alpha}z)\, |\psi(z)|^2\, dx - \int \limits_{\widetilde{\Sigma_\ve}} \varphi(\ve^{\alpha}z)\, |\psi(z)|^2\,\, d\sigma \Big|
\\[4mm]
\displaystyle
\le C\, \|\psi\|_{L^2(\widetilde{\Omega_\ve})}\,\|\nabla \psi\|_{L^2(\widetilde{\Omega_\ve})}.
\end{array}
$$
with some constant $C$ independent of $\ve$.
\end{lemma}
Lemma~\ref{lemma-1-bis} is proved in the same way as Lemma~\ref{lemma-1}.

\begin{lemma}
\label{lemma-2}
Suppose two nonnegative functions $f_1, f_2 \in C^3(\bar{B})$, defined on a bounded domain $B$, are such that $x=0$ is the global minimum point for both of them, and $f_1(0)=f_2(0)=0$. Moreover, assume that
$$
H(f_k)(0) \ge \alpha\, I, \quad \alpha >0,
$$
where $H(f_k)$ is the Hessian matrix of $f_k$, $k=1,2$.

Then there exists a constant $C$ such that
$$
C\, f_1 \le f_2 \le C^{-1}\, f_1.
$$
\end{lemma}
\begin{proof}
Assume that there exists a sequence $x_j \in B$ such that
$$
\frac{f_1(x_j)}{f_2(x_j)} \to 0, \,\, j \to \infty.
$$
Since $f_2$ is bounded, then $f_1(x_j) \to 0$, as $j \to \infty$. And, consequently, $x_j \to 0$, as $j \to \infty$. Recalling that $H(f_1)(0)$ is bounded from below, we arrive at contradiction. Lemma is proved.
\end{proof}

\begin{lemma}
\label{lemma-compact}
\textbf{Compactness result}\\
Denote
$$
H_Q^1(\mathbb{R}^d) = \Big\{w \in H^1(\mathbb{R}^d):\,\, \|v\|_Q^2 = \int \limits_{\mathbb{R}^d} |\nabla v|^2\, dz
+\int \limits_{\mathbb{R}^d} (z^T Q z)\, |v|^2\, dz < \infty \Big\}.
$$
Then $H_Q^1(\mathbb{R}^d)$ is compactly imbedded into $L^2(\mathbb{R}^d)$. In other words, any $\{v_n\} \subset H^1(\mathbb{R}^d)$ such that
$\|v_n\|_Q \le C$, converges strongly along a subsequence in $L^2(\mathbb{R}^d)$.
\end{lemma}
\begin{proof}
Obviously, $v_n$, up to a subsequence, converges weakly in $L^2(\mathbb{R}^d)$ to some function $v^\ast$, $n \to \infty$. Let us prove that $\|v_n\|_{L^2(\mathbb{R}^d)} \to \|v^\ast\|_{L^2(\mathbb{R}^d)}$, as $n \to \infty$.

Since
$$
\int \limits_{\mathbb{R}^d}(z^T Q z)|v_n|^2\, dz \le C,
$$
one can show that for any $\delta>0$, there exists a ball $B_{R(\delta)}(0)$ such that
$$
\int \limits_{\mathbb{R}^d \setminus B_{R(\delta)}(0)}|v_n|^2\, dz \le \delta.
$$
Without loss of generality we assume that $\|v_n\|_{L^2(\mathbb{R}^d)}=1$. Then
\be
\label{4}
\|v_n\|_{L^2(B_{R(\delta)}(0))}^2 = 1 - \|v_n\|_{L^2(\mathbb{R}^d)\setminus B_{R(\delta)}(0))}^2 \ge 1 -\delta^2.
\ee
Since $\|v_n\|_{H^1(B_{R(\delta)}(0))}\le C$, then $\|v_n - v^\ast\|_{L^2(B_{R(\delta)}(0))} \to 0$, as $n \to \infty$. Passing to the limit in \eqref{4}, we have
$$
\|v^\ast\|_{L^2(\mathbb{R}^d)} \ge \|v^\ast\|_{L^2(B_{R(\delta)}(0))}^2 \ge 1 - \delta^2.
$$
On the other hand,
$$
\|v^\ast\|_{L^2(\mathbb{R}^d)} \le \liminf \limits_{n \to \infty} \|v_n\|_{L^2(\mathbb{R}^d)}=1.
$$
Combining the last two inequalities yields $\|v^\ast\|_{L^2(\mathbb{R}^d)} = 1$. Lemma is proved.
\end{proof}

\begin{lemma}
\label{lemma-mean-value}
\textbf{Mean-value theorem}\\
Let $\Phi\in L^2(Y)$ be such that $\int_Y \Phi\, dy =0$, and $V \in C^1(\mathbb{R}^d)$ satisfy the estimate
\be
\label{5}
|D^k V(z)| \le C\, e^{-\gamma_0 |z|^2}, \,\, \gamma>0, \,\, k=0,1.
\ee
Denote by $\chi(x)$ a cut-off which is equal to $1$ if $|x|<\frac{1}{3}\,{\rm dist}(0, \partial \Omega)$, equal to $0$ if $|x|>\frac{1}{2}\,{\rm dist}(0, \partial \Omega)$, and is such that
\be
\label{6}
0 \leq \chi(x) \leq 1, \quad
|\nabla \chi| \leq C.
\ee
Then the following estimate holds:
$$
\Big|\int \limits_{\Omega_\ve} \Phi\big(\frac{x}{\ve}\big)\, V\big(\frac{x}{\ve^{\alpha}}\big)\, \chi\big(x\big)\, W\big( \frac{x}{\ve^\alpha} \big)\, dx\Big| \le C\,\ve^{1-\alpha}\, \ve^{d \alpha} \, \|\Phi\|_{L^2(Y)}\, \|W\|_{H^1(\mathbb{R}^d)}
$$
for any $W\in H^1(\mathbb{R}^d)$.
\end{lemma}
\begin{proof}
Since $\int_Y \Phi\, dy=0$, then there exists a periodic vector function $\varphi(y)$ such that
$$
\left\{
\ba{l}
\disp
-{\rm div}_y\, \varphi(y) = \Phi(y), \quad y \in Y,
\\[2mm]
(\varphi,n)=0, \quad y \in \Sigma^0,
\ea
\right.
$$
and $\|\varphi\|_{L^2(Y)} \le C\, \|\Phi\|_{L^2(Y)}$. Changing the variables we have
$$
-\ve\, {\rm div}\, \varphi\big(\frac{x}{\ve}\big) = \Phi\big(\frac{x}{\ve}\big).
$$
Multiplying the last equation by $V\big(\frac{x}{\ve^\alpha}\big)\, \chi_\ve\big(x\big)\, W\big( \frac{x}{\ve^\alpha} \big)$, integrating by parts over $\Omega_\ve$ and using \eqref{5}, \eqref{6}  we get
$$
\ba{l}
\disp
\Big|\int \limits_{\Omega_\ve} \Phi\big(\frac{x}{\ve}\big)\, V\big(\frac{x}{\ve^{\alpha}}\big)\, \chi\big(x\big)\, W\big( \frac{x}{\ve^\alpha} \big)\, dx\Big|
\\[4mm]
\disp
= \ve\, \Big|\int \limits_{\Omega_\ve} \varphi\big(\frac{x}{\ve}\big)\cdot \nabla \Big[V\big(\frac{x}{\ve^{\alpha}}\big)\, \chi\big(x\big)\, W\big( \frac{x}{\ve^\alpha} \big)\Big]\, dx\Big|
\\[4mm]
\disp
\le
C\, \ve^{1-\alpha}\, \ve^{d\, \alpha}\, \int \limits_{\widetilde{\Omega_\ve}} \big|\varphi\big(\frac{z}{\ve^{1-\alpha}}\big)\big|\,
e^{-\gamma_0 |z|^2}\, \big[|W| + |\nabla W|\big]\, dz
\\[4mm]
\disp
\le
C\, \ve^{1-\alpha}\, \ve^{d\, \alpha}\, \|W\|_{H^1(\mathbb{R}^d)}\, \Big(\int \limits_{\widetilde{\Omega_\ve}} \big|\varphi\big(\frac{z}{\ve^{1-\alpha}}\big)\big|^2\,e^{-2\gamma_0 |z|^2}\,dz\Big)^{1/2}
\\[4mm]
\disp
\le C\, \ve^{1-\alpha}\, \ve^{d\, \alpha}\, \|W\|_{H^1(\mathbb{R}^d)}\, \|\varphi\|_{L^2(Y)}\, \Big(\int \limits_{\mathbb{R}^d} e^{-2\gamma_0 |z|^2}\,dz\Big)^{1/2}.
\ea
$$
The integral in the parenthesizes converges. Lemma~\ref{lemma-mean-value} is proved.
\end{proof}


\end{document}